\documentclass[11pt]{amsart}
\usepackage{amsmath}
\usepackage{amssymb,latexsym}

\newcommand{\C}{{\mathbb  C}}
\newcommand{\D}{{\mathbb D}}

\newcommand{\R}{{\mathbb  R}}

\newcommand{\T}{{\mathbb  T}}
\newcommand{\A}{{\mathcal A}}

\newcommand{\cL}{{\mathcal L}}

\newcommand{\cO}{{\mathcal O}}
\newcommand{\boldcdot}{{\mathbf \cdot}}

\newcommand{\USC}{{\operatorname{{\mathcal{USC}}}}}
\newcommand{\PSH}{{\operatorname{{\mathcal{PSH}}}}}

\newtheorem{theorem}{Theorem}

\newtheorem{proposition}[theorem]{Proposition}

\theoremstyle{definition}

\begin{document}
\title[Weighted homogeneous Siciak-Zaharyuta extremal functions]
{A note on weighted homogeneous Siciak-Zaharyuta extremal functions}

\author[B. Drinovec Drnov\v sek]{Barbara Drinovec Drnov\v sek}
\address{Faculty of Mathematics and Physics\\ University of Ljubljana\\
Institute of Mathematics, Physics and Mechanics\\
Jadranska 19, 1000 Ljubljana, Slovenia}
\email{barbara.drinovec@fmf.uni-lj.si}
\thanks{B. Drinovec Drnov\v{s}ek was partially supported by grant P1-0291, Republic of Slovenia.}

\author[R. Sigursdsson]{Ragnar Sigurdsson}
\address{Department of Mathematics, School of Engineering and Natural Sciences\\
University of Iceland\\
IS-107 Reykjav\'ik, Iceland}
\email{ragnar@hi.is}

\date{\today}

\begin{abstract}
We prove that for any given upper semicontinuous function 
$\varphi$ on an open subset $E$ of $\C^n\setminus\{0\}$,
such that  the complex cone generated by $E$ minus 
the origin is connected, the homogeneous Siciak-Zaharyuta 
function with the weight $\varphi$ on $E$, can be represented as an
envelope of a disc functional. 
\end{abstract}




\maketitle

{\bf Introduction.}
Let  $\cL$ denote the Lelong class on $\C^n$ and 
$\cL^h$ the subclass of functions  
$u$ which are {\it logarithmically homogeneous}.
Let $\varphi\colon E\to \overline \R$ be a function on 
a subset $E$ of $\C^n$ taking values in the
extended real line $\overline \R$.
The {\it Siciak-Zaharyuta extremal
function $V_{E,\varphi}$ with weight $\varphi$} 
is defined by 
$$
V_{E,\varphi}=\sup \{u \in \cL\,;\, u|E \leq \varphi\}.
$$
The {\it homogeneous Siciak-Zaharyuta extremal
function $V_{E,\varphi}^h$ with weight $\varphi$}  is defined
similarly with $\cL^h$ in the role of $\cL$. 
In the special case when $\varphi=0$ we only write
$V_E$ (and $V_E^h$) and we call this function the 
({\it homogeneous}) {\it Siciak-Zaharyuta extremal function for the set $E$}.
The function $V_E$ ($V_E^h$) is also called the
({\it homogeneous}) {\it pluricomplex Green function for $E$ with pole
  at infinity}.

\begin{theorem}
  \label{th:main}
Let $\varphi\colon E\to \R\cup\{-\infty\}$ be an upper semicontinuous
function on an open subset $E$  of $\C^n\setminus\{0\}$. 
Assume that there exists a function in $\cL^h$ dominated by
$\varphi$ on $E$.  Then the 
largest logarithmically homogeneous function 
$\C E\to \R\cup\{-\infty\}$
dominated by $\varphi$ on $E$ is upper semicontinuous on $\C^* E$ and it
is of the form $\log\varrho_{E,\varphi}$, where
\begin{equation}
  \label{eq:1.0}
\varrho_{E,\varphi}(z)=\inf\{|\lambda|e^{\varphi(z/\lambda)} \,;\, 
\lambda\in \C^*, \ z/\lambda \in E\}, 
\qquad z\in \C^* E.
  \end{equation}
If  $\C^*E$ is connected, then for every $z\in \C^n$
\begin{align}
 \nonumber
  V_{E,\varphi}^h(z)
=\inf\Big\{  
\int_\T \log\varrho_{E,\varphi}(f_1,\dots,f_n) \, d\sigma \,;\, 
f\in \cO(\overline \D,\mathbb  P ^n),\,
   f=[f_0:\cdots:f_n],\\
f(\T)\subset \C^*E , \, 
f_0(0)=1, \, (f_1(0),\dots,f_n(0))=z. \Big\} \label{eq:1.1}
\end{align}
If  $\C E=\C^n$, then for every $z\in \C^n$
\begin{equation}
  \label{eq:1.2}
  V_{E,\varphi}^h(z)
=\inf\Big\{  
\int_\T \log\varrho_{E,\varphi}\circ f \, d\sigma \,;\, 
f\in \cO(\overline \D,\mathbb  C ^n),\,f(0)=z\Big\}.
\end{equation}
\end{theorem}

\medskip
{\it A disc envelope formula} is  a formula where
the values of a function $F$ defined on a complex
space $X$ with values on the extended real line $\overline \R$
are given as $F(z)=\inf\{H(f)\,;\, f\in {\mathcal B}(z)\}$, where
$H$ is {\it disc functional}, i.e., a function  
defined on some subset $\A$ of $\cO(\D,X)$, the set of 
{\it analytic
discs} in $X$, with values on $\overline \R$,
$\mathcal B$ is a subclass of $\A$, and
${\mathcal B}(z)$ consists of all of $f\in \mathcal B$ 
with {\it center} $z=f(0)$.

\smallskip
The formula (\ref{eq:1.1}) is an example of a disc envelope 
formula, where 
$\A$ consists of all closed analytic discs with value in 
the projective space, i.e., elements 
$f$ in $\cO(\overline \D,\mathbb  P ^n)$ which map the
unit circle $\T$ into $\C^*E$,
 $H(f)$ is the integral, and
${\mathcal B}$ is the subset of $\A$ consisting
of discs with $f_0(0)=1$.  We identify a point
$[1:z]\in \mathbb  P ^n$ with the point $z\in \C^n$.

\smallskip

For general information on the Siciak-Zaharyuta extremal function 
see 
Siciak \cite{Sic:1961,Sic:1962,Sic:1981,Sic:1982,Sic:2011}
and 
Zaharyuta \cite{Zah:1976}.
The first disc envelope formula for $V_E$ was proved by Lempert in the
case when $E$ is an open convex subset of $\C^n$ with real analytic
boundary.  (The proof is given in Momm \cite[Appendix]{Mom:1996}.)  
L\'arusson and Sigurdsson \cite{LarSig:2005} 
proved disc envelope formulas for $V_E$ for open connected
subsets $E$ of $\C^n$.  Magn\'usson and Sigurdsson 
\cite{MagSig:2007} generalized this result and obtained a disc formula 
for $V_{E,\varphi}$ in the case when $\varphi$ is an upper semicontinuous
function on an open connected subset $E$ of $\C^n$.
Drinovec Drnov\v{s}ek and Forstneri\v{c} \cite{DrnFor:2011a} proved disc envelope formulas for $V_E$ for open subsets $E$ of an irreducible and locally irreducible algebraic subvariety of $\C^n$.
Recently, Magn\'usson \cite{Mag:2013} established disc envelope formulas for the global extremal 
function in projective space.

\medskip





\bigskip
{\bf Acknowledgement.}  This paper was written 
while the second author was visiting University of Ljubljana in the
autumn of 2012.
He would like to thank the members of the Department of Mathematics
for their great hospitality and for many interesting and 
helpful discussions.

\bigskip
{\bf Notation.}
Let $\D$ denote the unit disc in $\C$, $\T$ the unit
circle, and $\sigma$ the arc length measure on $\T$ normalized to $1$.  
An analytic disc is a holomorphic map $f\colon \D\to X$, where 
$X$ is some complex space.  We let $\cO(\D,X)$ denote the set of all 
analytic discs.  We say that the disc is closed if it extends 
as a holomorphic map to some neighbourhood of the closed
unit disc $\overline \D$ with values in $X$ and we let 
$\cO(\overline \D,X)$ denote the set of all closed analytic discs in 
$X$.  The point $z=f(0)\in X$ is called the  center of $f$.

\smallskip
For a subset $X$ of $\C^n$ we let 
$\USC(X)$ denote the set of all upper semicontinuous functions on $X$,
and for open subset $U$ of $\C^n$ we  denote by $\PSH(U)$ the set of all plurisubharmonic functions on $U$.
The Lelong class $\cL$ consists of all $u\in \PSH(\C^n)$ such that
$u-\log^+|\boldcdot|$ is bounded above and 
$\cL^h$ is the subclass of all logarithmically homogeneous functions,
i.e., functions satisfying
$u(\lambda z)=u(z)+\log|\lambda|$ for 
$\lambda\in \C^*$ and $z\in \C^n$.
Observe that every such function takes the value
$-\infty$ at the origin.  
For every subset $E$ of $\C^n$ we set
$\C E=\{\lambda z\,;\, \lambda\in \C, z\in E\}$,
$\C^* E=\{\lambda z\,;\, \lambda\in \C^*, z\in E\}$
and we call $\C E$ the complex cone generated by $E$.
Note that complex cones are suitable sets for the domains of definition 
of logarithmically homogeneous functions.

\smallskip
Let $\mathbb  P ^n$ denote the $n$-dimensional projective space,
$\pi\colon \C^{n+1}\setminus\{0\}\to \mathbb  P ^n$ the natural projection,
$(z_0,\dots,z_n)\mapsto [z_0:\cdots:z_n]$, and identify 
$\C^n$ with the subset of all $[z_0:\cdots:z_n]$ with $z_0\neq 0$ and, in
particular, the point $z\in \C^n$ with $[1:z]\in \mathbb  P ^n$.  
{\it The hyperplane at infinity} is
$H_\infty=\pi(Z_0\setminus \{0\}\big)$, where $Z_0$ is the hyperplane
in $\C^{n+1}$ defined by the equation $z_0=0$.
Then $\mathbb  P ^n=\C^n\cup H_\infty$.

\bigskip
{\bf Review of a few results.}
Assume that $\psi\colon X\to \R\cup\{-\infty\}$ is a measurable function
on a subset $X$ of $\C^n$, such that there is 
$u\in \cL$ satisfying $u|X\leq \psi$.  
It is an easy observation that a function $u\in \PSH(\C^n)$ is
in $\cL$ if and only if the function 
\begin{equation*}
(z_0,\dots,z_n)\mapsto 
u(z_1/z_0,\dots,z_n/z_0)+\log|z_0|
\end{equation*}
extends as a plurisubharmonic function from $\C^{n+1}\setminus Z_0$ to
$\C^{n+1}\setminus \{0\}$.  Let $v$ denote this extension.
Take $f=[f_0:\cdots:f_n]\in \cO(\overline \D,\mathbb  P ^n)$ with
$f_0(0)=1$, $(f_1(0),\dots,f_n(0))=z$, satisfying $f(\T)\subset X$,  and define
$\tilde f=(f_0,\dots,f_n)\in \cO(\overline \D,\C^{n+1}\setminus\{0\})$.
Then the subaverage 
property of $v\circ \tilde f$ and the Riesz representation formula applied to
$\log|f_0|$ give (see \cite[p.~243]{MagSig:2007})
\begin{align*}
 u(z)&=u(f_1(0),\dots,f_n(0))+\log|f_0(0)|
=v\circ \tilde f(0) \\
&\leq \int_\T u(f_1/f_0,\dots,f_n/f_0)\, d\sigma+\int_\T \log|f_0|\, d\sigma\\
&\leq \int_\T \psi(f_1/f_0,\dots,f_n/f_0)\, d\sigma-\sum_{a\in f^{-1}(H_\infty)}m_{f_0}(a) \log|a|
.
\end{align*}
 
For an  open connected $X\subset \C^n$ and $\psi\in \USC(X)$,
Magn{\'u}sson and Sigurdsson
 \cite[Theorem~2]{MagSig:2007} proved that
for every $z\in \C^n$  
\begin{align}
 \nonumber
  V_{X,\psi}(z)=\inf\big\{
-\sum_{a\in f^{-1}(H_\infty)}m_{f_0}(a) \log|a|
+\int_\T \psi(f_1/f_0,\dots,f_n/f_0)\, d\sigma \,;\, \\
f\in \cO(\overline \D,\mathbb  P ^n),\ f(\T)\subset X, \ f_0(0)=1, \ 
(f_1(0),\dots,f_n(0))=z  \label{eq:2.3}
\big\}.
\end{align}

\bigskip
Our main result, Theorem~\ref{th:main}, will follow from this formula
 and the following

\begin{proposition}   \label{prop:2.1}
Let $\varphi\colon E\to \R\cup\{-\infty\}$ be a function
on a subset $E \subset \C^n\setminus\{0\}$
such that there exists $u\in \cL^h$ satisfying
$u|E\leq \varphi$.  Let $\tilde \varphi\colon \C E\to \R\cup\{-\infty\}$
be the supremum of all  logarithmically homogeneous functions
on $\C E$  dominated by $\varphi$ on $E$. 
Then the following hold:
  \begin{enumerate}
  	\item[(i)] 
  	$\tilde \varphi$ is logarithmically homogeneous on $\C E$ and
  	for every $z\in \C^* E$ 
		\begin{equation}	\label{eq:prop(i)}
			\tilde \varphi(z)=\inf \{ \varphi(\lambda z)-\log |\lambda|
			\,;\, \lambda\in \C^* \text{ and } \lambda z \in E \},
		\end{equation}	 
  	\item[(ii)] 
			$V_{E,\varphi}^h=V_{E,\tilde\varphi}^h
			=V_{\C^* E,\tilde \varphi}^h$.
  \end{enumerate}
  If, in addition to the above, $\C^* E$ is nonpluripolar and $\varphi\in \USC(E)$ then 
  \begin{enumerate}
    \item[(iii)] 
				$\tilde \varphi\in \USC(\C^* E)$ and $V_{E,\varphi}^h=V_{\C^* E,\tilde \varphi}$,
		\item[(iv)] 
		  if $\C E=\C^n$, then 	$\tilde \varphi\in \USC(\C^n)$ and 
		  $$V_{E,\varphi}^h=\sup \{u \in \PSH(\C^n)\,;\, u \leq \tilde\varphi\}.$$
  \end{enumerate}
\end{proposition}

\begin{proof} 
(i)
It is easy to see that the supremum of any family of logarithmically homogeneous functions defined on a complex cone is a logarithmically homogeneous function provided the family
is bounded from above at any point of the cone. Take $z\in \C^* E$ and choose
$\lambda\in \C^*$ such that $\lambda z\in E$. For any logarithmically homogeneous
function $u$ on $\C E$  dominated by $\varphi$ on $E$ we have
\begin{equation}
\label{eq:proofprop}
u(z)=u(\lambda z)-\log |\lambda|\le \varphi(\lambda z)-\log |\lambda|
\end{equation}
which implies that the family is bounded from above at $z$. Since all 
logarithmically homogeneous functions take the value $-\infty$ at the origin
the family is bounded from above at any point of the cone.
 
 Let $\psi$ denote the function on $\C^*E$ whose
  value at $z$ is given by the right hand side of the equation (\ref{eq:prop(i)}).
For a  logarithmically homogeneous function $u$ on $\C E$,
dominated by $\varphi$ on $E$, we have $u(z)\le \varphi(\lambda z)-\log |\lambda|$
for any $\lambda\in \C^*$ such that $\lambda z\in E$ by (\ref{eq:proofprop}).
Taking infimum over all $\lambda\in\C^*$ with $\lambda z\in E$ shows that 
$u\leq \psi$ on $\C^* E$. Hence $\tilde \varphi \leq \psi$ on $\C^* E$.
To prove the converse inequality note that
  \begin{align}\label{eq:prop1}
  \psi(\mu z)&=\inf \{ \varphi(\lambda\mu z)-\log |\lambda| 
			\,;\, \lambda\in \C^* \text{ and } \lambda\mu z \in E \}\\ \nonumber
							&=\inf \{ \varphi(\lambda\mu z)-\log |\lambda\mu|
			\,;\, \lambda\in \C^* \text{ and } \lambda\mu z \in E \}+\log |\mu|\\ \nonumber
						&=\psi(z)+\log |\mu|
	\end{align}		
for any $z\in\C^*E$ and $\mu\in \C^*$ 
thus the map $\psi$ is logarithmically homogeneous. 
Since $\psi\le \varphi$ on $E$ we get $\psi\le \tilde \varphi$.
  
\smallskip
(ii)  Since $\varphi\geq \tilde\varphi$ on $E$ and
$E\subset \C^* E$ we have
$V_{E,\varphi}^h\geq 
V_{E,\tilde\varphi}^h
\geq V_{\C^* E,\tilde \varphi}^h$.
For proving  the two equalities we
take $u\in \cL^h$ with $u|E\leq \varphi$.
By (i) we obtain
$u\leq \tilde \varphi$ on $\C^* E$ which implies
$V_{\C^* E,\tilde \varphi}^h\geq V_{E,\varphi}^h$. 

\smallskip
(iii)
To prove that $\tilde \varphi$ is upper semicontinuous
take $z _0\in \C^* E$ and $c>\tilde \varphi(z _0)$.
We need to show that $c>\tilde \varphi(z )$ for all 
$z $ in some neighbourhood $U$ of $z _0$.  We choose
$\lambda_0\in \C^*$ such that $\lambda_0z _0\in E$
and such that $c>\varphi(\lambda_0z _0)-\log|\lambda_0|$.  
Since $\varphi\in \USC(E)$ there
exists an open neighbourhood $U$ of $z _0$
such that $\lambda_0z\in E $ and 
$c>\varphi(\lambda_0z )-\log|\lambda_0|$ 
for all $z \in U$.  By (i) we have
$c>\tilde \varphi(z )$ for all $z \in U$. 
 
Since $\cL^h\subset \cL$ we have
$V_{\C^* E,\tilde \varphi}^h \leq V_{\C^* E,\tilde \varphi}$.
For proving the opposite inequality
we take $u\in \cL$ such that $u\leq \tilde \varphi$ on $\C^* E$.
Then $u(\lambda z )-\log|\lambda|\leq \tilde\varphi(\lambda z )-\log|\lambda|=\tilde\varphi( z )$
for all $ z \in \C^* E$ and $\lambda\in \C^*$.
Let $v$ be the upper semicontinous regularization
of the function 
$\sup\{u(\lambda\boldcdot)-\log|\lambda| \,;\, \lambda\in \C^*\}$ on $\C^n$.  
We have  $u\leq v\leq \tilde \varphi$ on $\C^* E$ and since $\C^*E$ is
nonpluripolar and $\tilde\varphi$ is locally bounded above on  $\C^*E$,
we have $v\in \cL$. 
A similar calculation as in (\ref{eq:prop1})
shows that $v$ is logarithmically homogeneous, which proves the opposite inequality.

\smallskip
(iv)
The fact that $\tilde \varphi$ is upper semicontinuous at $0$ easily follows from
the fact that $\tilde \varphi$ is bounded from above on the unit sphere and that it is
logarithmically homogeneous. By (iii) we get $V_{E,\varphi}^h=V_{\C^* E,\tilde \varphi}$ and it is easy to see that in the case $\C E=\C^n$ the latter equals
$V_{\C^n,\tilde \varphi}$.

Let $P_{\tilde \varphi}$ denote the function whose
value at $z$ is given by the right hand side of the equation.
Since $\cL\subset \PSH(\C^n)$ it follows 
$V_{\C^n,\tilde \varphi} \leq P_{\tilde\varphi}$.
 To prove the opposite inequality,
it is enough to show that $P_{\tilde \varphi} \in \cL$.
Since $\tilde \varphi\subset \USC(\C^n)$ it follows that  $P_{\tilde \varphi} $
is the largest plurisubharmonic function on $\C^n$ dominated by $\tilde\varphi$.
By upper semicontinuity the map $\tilde \varphi$ is bounded from above 
on the unit sphere in $\C^n$ by some constant $M\in \R$. 
Since $\tilde\varphi$ is logarithmically homogeneous we get
$$P_{\tilde \varphi}(\lambda z)\le \tilde\varphi(\lambda z)\leq
\log|\lambda|+M=\log|\lambda z|+M$$
for any $z\in\C^n$, $|z|=1$, and $\lambda\in\C^*$. 
It follows that $P_{\tilde \varphi} \in \cL$.
\end{proof}

\smallskip
{\it Proof of Theorem \ref{th:main}.}
By Proposition \ref{prop:2.1} 
the largest logarithmically homogeneous function 
$\tilde \varphi\colon\C E\to \R\cup\{-\infty\}$
dominated by $\varphi$ on $E$ is upper semicontinuous on $\C^* E$
and $\varrho_{E,\varphi}=e^{\tilde \varphi(z)}
=\inf\{ |\lambda|e^{\varphi(z/\lambda)} \,;\, 
\lambda\in \C^*, \ z/\lambda\in E\}$ which proves (\ref{eq:1.0}).

If we take $X=\C^* E$ and $\psi=\tilde \varphi$ in (\ref{eq:2.3}), 
then logarithmic homogeneity
of $\tilde \varphi$ on $\C^* E$ implies that 
$$
\int_\T \tilde\varphi(f_1/f_0,\dots,f_n/f_0)\, d\sigma
=\int_\T \tilde\varphi(f_1,\dots,f_n)\, d\sigma
-\int_\T \log|f_0| \, d\sigma.
$$ 
If $f_0(0)=1$, then the Riesz representation formula gives
$$
\sum_{a\in f^{-1}(H_\infty)}m_{f_0}(a) \log|a|
+
\int_\T \log|f_0| \, d\sigma=0,
$$
which implies that the right hand side of 
(\ref{eq:2.3}) reduces to
\begin{align*}
  V_{\C^* E,\tilde\varphi}(z)=\inf\Big\{&
\int_\T \tilde\varphi(f_1,\dots,f_n) \, d\sigma \,;\, 
f\in \cO(\overline \D,\mathbb  P ^n),\\ 
& f(\T)\subset \C^*E , \ 
f_0(0)=1,\ (f_1(0),\dots,f_n(0))=z
\Big\},
\end{align*}
thus (\ref{eq:1.1}) follows from Proposition \ref{prop:2.1} (iii).

If $\C E=\C^n$ then Proposition \ref{prop:2.1} (iv)  and
Poletsky theorem \cite{Pol:1991,Pol:1993}
imply
\begin{align*}
V_{E,\varphi}^h
&=\sup \{u \in \PSH(\C^n)\,;\, u \leq \tilde\varphi\}\\
&=\inf\Big\{  \int_\T \log\varrho_{E,\varphi}\circ f \, d\sigma \,;\, 
f\in \cO(\overline \D,\mathbb  C ^n),\,f(0)=z\Big\}
\end{align*}
which proves (\ref{eq:1.2}).
\hfill$\square$

\bigskip
{\bf Observation.} \ In the special case $\varphi=0$
we write $\varrho_E$ for $\varrho_{E,\varphi}$.
The function  $\varrho_E$ is absolutely homogeneous of degree $1$,
i.e., $\varrho_E(z\zeta)=|z|\varrho_E(\zeta)$.
Thus, if $E$ is a balanced domain, i.e., $\overline\D E=E$, then 
$\varrho_E$ is its Minkowski function.

{\small
\bibliographystyle{plain}
\bibliography{bibref}
\end{document}
\smallskip\noindent
Faculty of Mathematics and Physics, University of Ljubljana, and\\
Institute of Mathematics, Physics, and Mechanics, \\
Jadranska 19, 1000 Ljubljana, SLOVENIA.\\
barbara.drinovec@fmf.uni-lj.si

\medskip\noindent
and

\medskip\noindent
Department of Mathematics, School of Engineering and Natural Sciences, 
University of Iceland,\\
IS-107 Reykjav\'ik, ICELAND. \\
ragnar@hi.is
}